\documentclass[12pt]{amsart}
\usepackage{amssymb}
\usepackage{amssymb,amsfonts,amsmath,amsthm,cite,verbatim}
\usepackage{xcolor}
\usepackage{tikz}
\usepackage{graphicx}
\usepackage[left=2.5cm, right=2.5cm, top=3.5cm, bottom=3.5cm]{geometry}

\usepackage{hyperref}

\usepackage[normalem]{ulem}
\newtheorem{theorem}{Theorem}[section]
\newtheorem{thm}[theorem]{Theorem}
\newtheorem{prop}[theorem]{Proposition}
\newtheorem{lem}[theorem]{Lemma}

\newtheorem{cor}[theorem]{Corollary}
\newtheorem{ex}[theorem]{Example}

\makeatletter \@addtoreset{equation}{section}

\newcommand{\qbinom}[2]{\genfrac{[}{]}{0pt}{}{#1}{#2}}

\DeclareMathOperator*{\CT}{CT}

\begin{document}

\title[Symmetric function generalization of the $q$-Dyson theorem]{A symmetric function generalization of the Zeilberger--Bressoud $q$-Dyson theorem}

\author{Yue Zhou}

\address{School of Mathematics and Statistics, Central South University,
Changsha 410075, P.R. China}

\email{zhouyue@csu.edu.cn}


\date{August 26, 2020}

\begin{abstract}
In 2000, Kadell gave an orthogonality conjecture for a symmetric function generalization of the Zeilberger--Bressoud $q$-Dyson theorem or the $q$-Dyson constant term identity.
This conjecture was proved by K\'{a}rolyi, Lascoux and Warnaar in 2015.
In this paper, by slightly changing the variables of Kadell's conjecture, we obtain another symmetric function generalization of the $q$-Dyson constant term identity. This new generalized constant term admits a simple product-form expression.

\end{abstract}

\maketitle

\section{Introduction}\label{sec-intr}

In 1975, Andrews \cite{andrews1975} formulated a $q$-analogue conjecture for the Dyson constant term identity \cite{dyson}. For nonnegative integers $a_1,a_2,\dots,a_n$,
\begin{equation}\label{q-Dyson}
\CT_x \prod_{1\leq i<j\leq n}
(x_i/x_j;q)_{a_i}(qx_j/x_i;q)_{a_j}=
\frac{(q;q)_{a_1+\cdots+a_n}}{(q;q)_{a_1}(q;q)_{a_2}\cdots(q;q)_{a_n}},
\end{equation}
where $\displaystyle\CT_{x}$ denotes taking the constant term with respect to $x:=(x_1,\dots,x_n)$, and
$(z;q)_k:=(1-z)(1-zq)\dots(1-zq^{k-1})$ is a
$q$-shifted factorial for $k$ a positive integer and $(z;q)_0:=1$.
When $q\rightarrow 1$ the identity~\eqref{q-Dyson} reduces to the Dyson constant term identity.

In 1985, Zeilberger and Bressoud \cite{zeil-bres1985} gave the first proof of
Andrews' $q$-Dyson conjecture using tournaments.
Twenty years later Gessel and Xin \cite{gess-xin2006} gave a second proof using formal Laurent series, and
then, in 2014, K\'{a}rolyi and Nagy \cite{KN} discovered a very short and elegant proof using multivariable Lagrange interpolation.
Finally, Cai~\cite{cai} found an inductive proof by adding additional parameters to the problem. These days, Andrews' ex-conjecture is usually referred as the Zeilberger--Bressoud $q$-Dyson theorem or the $q$-Dyson constant term identity.

For $r$ a nonnegative integer, the $r$th complete symmetric function $h_r(x)$
may be defined in terms of its generating function as
\begin{equation}\label{e-gfcomplete}
\sum_{r\geq 0} z^r h_r(x)=\prod_{i=1}^n
\frac{1}{1-zx_i}.
\end{equation}
More generally, for the complete symmetric function indexed by a partition
$\lambda=(\lambda_1,\lambda_2,\dots)$
\[
h_{\lambda}:=h_{\lambda_1} h_{\lambda_2}\cdots.
\]
For $a:=(a_1,a_2,\dots,a_n)$ a sequence of nonnegative integers, let $x^{(a)}$ denote the alphabet
\[
x^{(a)}:=(x_1,x_1q,\dots,x_1q^{a_1-1},\dots,
x_n,x_nq,\dots,x_nq^{a_n-1})
\]
of cardinality $|a|:=a_1+\cdots+a_n$.
In \cite{kadell}, Kadell defined a generalized $q$-Dyson constant term
\begin{equation}\label{GDyson1}
\widetilde{D}_{v,\lambda}(a):=\CT_x
x^{-v}h_{\lambda}\big(x^{(a)}\big)
\prod_{1\leq i<j\leq n}
(x_i/x_j;q)_{a_i}(qx_j/x_i;q)_{a_j}.
\end{equation}
Here $v=(v_1,v_2,\dots,v_n)\in\mathbb{Z}^{n}$,
$x^v$ denotes the monomial $x_1^{v_1}\cdots x_n^{v_n}$,
and $\lambda$ is a partition such that $|v|=|\lambda|$.
He \cite[Conjecture 4]{kadell} conjectured that
for $r$ a positive integer and $v$ a weak composition such that $|v|=r$,
\begin{equation}\label{kadellconj}
\widetilde{D}_{v,(r)}(a)=
\begin{cases}
\displaystyle
\frac{q^{\sum_{i=k+1}^n a_i}(1-q^{a_k})(q^{|a|+1};q)_{r-1}}{(q^{|a|-a_k+1};q)_r}
\prod_{i=1}^n\qbinom{a_i+\cdots+a_n}{a_i}
& \text{if $v=(0^{k-1},r,0^{n-k})$}, \\[6mm]
0 & \text{otherwise},
\end{cases}
\end{equation}
where $\qbinom{n}{k}=(q^{n-k+1};q)_k/(q;q)_k$ is a $q$-binomial coefficient.
In fact Kadell only considered $v=(r,0^{n-1})$ in his conjecture, but the
more general statement given above is what was proved by K\'{a}rolyi,
Lascoux and Warnaar in \cite[Theorem 1.3]{KLW} using multivariable Lagrange
interpolation and key polynomials.
For a sequence $u=(u_1,\dots,u_n)$ of integers, we denote
by $u^{+}$ the sequence obtained from $u$ by ordering the $u_i$
in weakly decreasing order (so that $u^{+}$ is a partition
if $u$ is a composition). K\'{a}rolyi et al.\ also
proved a closed-form expression for
$\widetilde{D}_{v,v^{+}}(a)$ in the case when $v$ is a weak composition all of whose parts
are distinct, i.e., $v_i\neq v_j$ for all $1\leq i<j\leq n$.
Subsequently, Cai~\cite{cai} gave an inductive proof of Kadell's conjecture.
Recently, we got a recursion for $\widetilde{D}_{v,v^+}(a)$ if
$v$ is a weak composition with unique largest part~\cite{Zhou}.
Later, we obtained a recursion for $\widetilde{D}_{v,v^+}(a)$ for arbitrary nonzero weak composition $v$ in~\cite{Zhou2}.
If $v$ is a weak composition with equal nonzero parts, for example $v=(2,2,0^{n-2})$, it appears that the expression for $\widetilde{D}_{v,v^+}(a)$ is no longer a product. Hence, by slightly changing the variables in
$\widetilde{D}_{v,\lambda}(a)$, we give a new constant term $D_{v,\lambda}(a)$ (see \eqref{GDyson} below), which admits a simple product-form formula for the $v=\lambda$ case.

For $i\in \{1,\dots,n\}$ and $a:=(a_1,a_2,\dots,a_n)$ a sequence of nonnegative integers, let $x^{(a)}_i$ denote the alphabet
\begin{multline}\label{alphabet-x}
x^{(a)}_i:=(x_1,x_1q,\dots,x_1q^{a_1-1},\dots,x_{i-1},x_{i-1}q,\dots,x_{i-1}q^{a_{i-1}-1},\\
x_iq^{-1},x_i,\dots,x_iq^{a_i-1},x_{i+1},x_{i+1}q,\dots,x_{i+1}q^{a_{i+1}-1},\dots,
x_n,x_nq,\dots,x_nq^{a_n-1}),
\end{multline}
of cardinality $|a|+1$. Note that $x^{(a)}_i$ is the alphabet obtained by adding the variable $x_iq^{-1}$
to $x^{(a)}$.
Define another generalized $q$-Dyson constant term
\begin{equation}\label{GDyson}
D_{v,\lambda}(a):=\CT_x
x^{-v}\prod_{i=1}^nh_{\lambda_i}\big(x_i^{(a)}\big)
\prod_{1\leq i<j\leq n}
(x_i/x_j;q)_{a_i}(qx_j/x_i;q)_{a_j},
\end{equation}
where $v=(v_1,v_2,\dots,v_n)\in\mathbb{Z}^{n}$ and $\lambda=(\lambda_1,\lambda_2,\dots,\lambda_n)$ is a partition such that $|v|=|\lambda|$. (Note that if $|v|\neq |\lambda|$ then $D_{v,\lambda}(a)=0$.)
For $u,v\in \mathbb{Z}^{n}$, define $u\preceq v$ if either $u=v$, or else for some $i$,
\[
u_1=v_1,\dots,u_{i-1}=v_{i-1},u_{i}<v_{i},\quad u_{j}\leq v_{i} \ \text{for all $j>i$.}
\]
We write $u\prec v$ if $u\preceq v$ but $u\neq v$.
In this paper, we are concerned with the $v\preceq\lambda$ case of $D_{v,\lambda}(a)$.
We obtain a closed-form formula for $D_{\lambda,\lambda}(a)$ if $\lambda$ is a  partition with length at most $n$. In addition, we find that the constant term $D_{v,\lambda}(a)$ vanishes for $v\prec \lambda$.
These results are the content of the next theorem.
\begin{thm}\label{thm-1}
Let $D_{v,\lambda}(a)$ be defined in \eqref{GDyson}. Then
\begin{equation}
D_{v,\lambda}(a)=
\begin{cases}
\displaystyle
q^{-|\lambda|}\prod_{i=1}^n\qbinom{a_i+\cdots+a_n+\lambda_i}{a_i}
&\text{if $v=\lambda$,}\\[4mm]
0 &\text{if $v\prec \lambda$.}
\end{cases}
\end{equation}
\end{thm}
Note that for $v=\lambda=(0,\dots,0)$, Theorem~\ref{thm-1} reduces to the $q$-Dyson constant term identity~\eqref{q-Dyson}, but Kadell's ex-conjecture~\eqref{kadellconj} does not.
It is easy to see that for partitions of the same size the order $\preceq$ corresponds to the reverse lexicographic order $\mathop\leq\limits^{R}$.
For $v$ a composition, if $v^+\mathop{<}\limits^{R}\lambda$ then
$v^+\prec \lambda$ and $v\prec \lambda$.
Hence, we obtain the following corollary.
\begin{cor}
Let $D_{v,\lambda}(a)$ be defined in \eqref{GDyson}. If $v$ is a weak composition such that $v^+\mathop{<}\limits^{R}\lambda$, then $D_{v,\lambda}(a)=0$.
\end{cor}

To prove Theorem~\ref{thm-1}, we adopt the idea of Cai's proof \cite{cai} of Kadell's orthogonality (ex)-conjecture.
For $w=(w_1,\dots,w_n)$ a sequence of parameters, denote
\begin{equation}\label{d-F}
F(a,w):=\prod_{1\leq i<j\leq n}
(x_i/x_j;q)_{a_i}(qx_j/x_i;q)_{a_j}
\prod_{i=1}^n\prod_{j=1}^n(q^{-\chi(j=i)}x_i/w_j;q)_{a_i+\chi(j=i)}^{-1},
\end{equation}
where $\chi(true)=1$ and $\chi(false)=0$.
Throughout this paper, we assume that all factors of the form $cx_i/w_j$ in \eqref{d-F}
satisfy $|cx_i/w_j|<1$, where $c\in \mathbb{C}(q)\setminus \{0\}$. Then
\[
\frac{1}{1-cx_i/w_j}=\sum_{k\geq 0} (cx_i/w_j)^k.
\]
Together with the generating function of complete symmetric functions \eqref{e-gfcomplete},
the constant term $D_{v,\lambda}(a)$ equals a certain coefficient of $F(a,w)$, that is
\begin{equation}\label{e-relation}
D_{v,\lambda}(a)=\CT_{x,w}x^{-v}w^{\lambda}F(a,w).
\end{equation}
By partial fraction decomposition, we obtain a splitting formula for $F(a,w)$ (see Lemma~\ref{lem-split} below). Using the splitting formula, we can obtain certain coefficients of $F(a,w)$.

The remainder of this paper is organised as follows.
In the next section we introduce some basic
notation used throughout this paper.
In Section~\ref{sec-split} we obtain a splitting formula for $F(a,w)$.
In Section~\ref{sec-proof} we utilize the splitting formula to give a proof of Theorem~\ref{thm-1}. In Section~\ref{sec-beyond} we discuss a relation between $\widetilde{D}_{v,\lambda}(a)$ and $D_{v,\lambda}(a)$, and some examples of $D_{v,\lambda}(a)$ for $v\npreceq \lambda$.

\section{Basic notation}

In this section we introduce some basic notation
used throughout this paper.

For $v=(v_1,\dots,v_n)$ a sequence, we write
$|v|$ for the sum of its entries, i.e.,
$|v|=v_1+\cdots+v_n$.
Moreover, if $v\in\mathbb{R}^{n}$ then we write
$v^{+}$ for the sequence obtained from $v$
by ordering its elements in weakly decreasing order.
If all the entries of $v$ are positive (resp. nonnegative) integers,
we refer to $v$ as a (resp. weak) composition.
A partition is a sequence $\lambda={(\lambda_1,\lambda_2,\dots)}$ of nonnegative integers such that
${\lambda_1\geq \lambda_2\geq \cdots}$ and
only finitely-many $\lambda_i$ are positive.
The length of a partition $\lambda$, denoted
$\ell(\lambda)$ is defined to be the number of nonzero $\lambda_i$ (such $\lambda_i$ are known as the parts of $\lambda$).
We adopt the convention of not displaying the
tails of zeros of a partition.
We say that $|\lambda|=\lambda_1+\lambda_2+\cdots$ is the size of the partition $\lambda$.
We adopt the standard dominance order on
the set of partitions of the same size.
If $\lambda$ and $\mu$ are partitions such that $|\lambda|=|\mu|$ then $\lambda\leq \mu$ if
$\lambda_1+\cdots+\lambda_i\leq \mu_1+\cdots+\mu_i$ for all $i\geq 1$.
A linear extension of the dominance order is the reverse lexicographic order, denoted $\overset{R}{\leq}$.
Given $|\lambda|=|\mu|$, define
$\lambda\overset{R}{\leq}\mu$ if either $\lambda=\mu$ or for some $i$,
\[
\lambda_1=\mu_1,\dots,\lambda_{i-1}=\mu_{i-1},\quad \lambda_{i}<\mu_{i}.
\]
Let $u=(u_1,\dots,u_m)$ and $v=(v_1,\dots,v_n)$ be two integer sequences. We can set $u_j=0$ for $j>m$ and $v_i=0$ for $i>n$, and do not require $|u|=|v|$.
Define $u\preceq v$ if either $u=v$ or for some $i$,
\begin{equation}\label{e-defiprec}
u_1=v_1,\dots,u_{i-1}=v_{i-1},u_{i}<v_{i},\quad u_{j}\leq v_{i} \ \text{for all $j>i$.}
\end{equation}
If $u$ and $v$ are partitions such that $|u|=|v|$, then the conditions $u_j\leq v_i$ for all $j>i$ in \eqref{e-defiprec} is redundant, since it is guaranteed by $u_i<v_i$ and $u_i\geq u_j$ for all $j>i$.
Hence, the order $\preceq$ is just reverse lexicographic order for partitions of the same size.
As usual, we write $\lambda<\mu$ if $\lambda\leq \mu$ but $\lambda\neq \mu$, $\lambda\overset{R}{<}\mu$ if $\lambda\overset{R}{\leq} \mu$
but $\lambda\neq \mu$, and $u\prec v$ if $u\preceq v$ but $u\neq v$.

For $k$ an integer such that $1\leq k\leq n$, define
\[
v^{(k)}:=(v_1,\dots,v_{k-1},v_{k+1},\dots,v_n).
\]

For $k$ a nonnegative integer,
\[
(z)_k=(z;q)_k:=
\begin{cases}
(1-z)(1-zq)\cdots (1-zq^{k-1})\quad &\text{if $k>0$,}\\
1 \quad &\text{if $k=0$,}
\end{cases}
\]
where, typically, we suppress the base $q$.
Using the above we can define the
$q$-binomial coefficient as
\[
\qbinom{n}{k}=\frac{(q^{n-k+1})_k}{(q)_k}
\]
for $n$ and $k$ nonnegative integers.

\section{A splitting formula for $F(a,w)$}\label{sec-split}

In this section, we give a splitting formula for $F(a,w)$.
To prove the splitting formula we need the next simple result.
\begin{lem}\label{prop-c}
Let $k$ be an integer, and $i,j$ be nonnegative integers. Then,
\begin{subequations}
\begin{equation}\label{prop-a}
\frac{(z)_j(q/z)_i}{(q^{-k}/z)_i}=q^{(k+1)i}(q^{-i}z)_{k+1}(q^{k+1}z)_{j-k-1}
\quad \text{for $-1\leq k\leq j-1$,}
\end{equation}
\begin{equation}\label{prop-b1}
\frac{(1/z)_i(qz)_j}{(q^{-k-1}/z)_{i+1}}=-q^{(i+1)k+1}z(q^{1-i}z)_k(q^{k+2}z)_{j-k-1} 
\quad \text{for $j>0$ and $0\leq k\leq j-1$,}
\end{equation}
and
\begin{equation}\label{prop-b2}
\frac{(1/z)_{i}(qz)_j}{(q^{-k}/z)_i}=q^{ik}(q^{1-i}z)_k(q^{k+1}z)_{j-k} \quad \text{for $0\leq k\leq j$.}
\end{equation}
\end{subequations}
\end{lem}
Note that the $j=k$ case of \eqref{prop-b2} (taking $z\mapsto z/q$) is the standard fact in \cite[Equation~(I.13)]{GR}.
\begin{proof}
For $-1\leq k\leq j-1$,
\begin{align*}
\frac{(z)_j(q/z)_i}{(q^{-k}/z)_i}=\frac{(z)_j(q^{i-k}/z)_{k+1}}{(q^{-k}/z)_{k+1}}
&=\frac{(-1/z)^{k+1}q^{(k+1)i-\binom{k+1}{2}}(q^{-i}z)_{k+1}(z)_j}
{(-1/z)^{k+1}q^{-\binom{k+1}{2}}(z)_{k+1}}\\
&=q^{(k+1)i}(q^{-i}z)_{k+1}(q^{k+1}z)_{j-k-1}.
\end{align*}
For $j>0$ and $0\leq k\leq j-1$,
\begin{align*}
\frac{(1/z)_i(qz)_j}{(q^{-k-1}/z)_{i+1}}
=\frac{(q^{i-k}/z)_{k}(qz)_j}{(q^{-k-1}/z)_{k+1}}
&=\frac{(-1/z)^kq^{ik-\binom{k+1}{2}}(q^{1-i}z)_k(qz)_j}
{(-1/z)^{k+1}q^{-\binom{k+2}{2}}(qz)_{k+1}}\\
&=-zq^{(i+1)k+1}(q^{1-i}z)_k(q^{k+2}z)_{j-k-1}.
\end{align*}
For $0\leq k\leq j$,
\[
\frac{(1/z)_{i}(qz)_j}{(q^{-k}/z)_i}=\frac{(q^{i-k}/z)_k(qz)_{j}}{(q^{-k}/z)_{k}}
=\frac{(-1/z)^kq^{ik-\binom{k+1}{2}}(q^{1-i}z)_k(qz)_j}{(-1/z)^{k}q^{-\binom{k+1}{2}}(qz)_k}
=q^{ik}(q^{1-i}z)_k(q^{k+1}z)_{j-k}.\qedhere
\]
\end{proof}

By partial fraction decomposition with respect to $w_1$, the rational function $F(a,w)$ admits the following partial fraction expansion.
\begin{lem}\label{lem-split}
Let $F(a,w)$ be defined in \eqref{d-F}, and set $I:=\{i\mid a_i\neq 0,\ i=2,\dots,n\}$. Then,
\begin{equation}\label{e-Fsplit}
F(a,w)=\sum_{k=-1}^{a_1-1}\frac{A_k}{1-q^kx_1/w_1}
+\sum_{i\in I}\sum_{j=0}^{a_i-1}\frac{B_{ij}}{1-q^jx_i/w_1},
\end{equation}
where
\begin{equation}\label{A}
A_k=\frac{\prod_{i=2}^nq^{(k+1)a_i}\big(q^{-a_i}x_1/x_i\big)_{k+1}
\big(q^{k+1}x_1/x_i\big)_{a_1-k-1}
\big(x_1/w_i\big)_{a_1}^{-1}}
{(q^{-k-1})_{k+1}(q)_{a_1-k-1}}\times F(a^{(1)},w^{(1)}),
\end{equation}
and
\begin{align}\label{B}
B_{ij}&=\frac{-q^{(a_1+1)j+1}x_i}{(q^{-j})_j(q)_{a_i-j-1}x_1}
\big(q^{1-a_1}x_i/x_1\big)_j\big(q^{j+2}x_i/x_1\big)_{a_i-j-1}
\prod_{l=2}^{i-1}q^{ja_l}\big(q^{1-a_l}x_i/x_l\big)_j\big(q^{j+1}x_i/x_l\big)_{a_i-j}\\
&\qquad\times
\prod_{l=i+1}^{n}q^{(j+1)a_l}\big(q^{-a_l}x_i/x_l\big)_{j+1}\big(q^{j+1}x_i/x_l\big)_{a_i-j-1}
\prod_{\substack{1\leq v<u\leq n\\v,u\neq i}}
(x_v/x_u)_{a_v}(qx_u/x_v)_{a_u} \nonumber\\
&\qquad \times
\prod_{u=1}^n\prod_{v=2}^n(q^{-\chi(u=v)}x_u/w_v)_{a_u+\chi(u=v)}^{-1}.\nonumber
\end{align}
\end{lem}
Note that $A_k$ is a power series in $x_1$, and $B_{ij}$ is a power series in $x_i$ with no constant term.
\begin{proof}
By partial fraction decomposition of $F(a,w)$ with respect to $w_1$, we can rewrite $F(a,w)$ as \eqref{e-Fsplit} and
\begin{subequations}
\begin{equation}\label{subs-a}
A_k=F(a,w)(1-q^kx_1/w_1)|_{w_1=q^kx_1},
\end{equation}
and
\begin{equation}\label{subs-b}
B_{ij}=F(a,w)(1-q^jx_i/w_1)|_{w_1=q^jx_i}.
\end{equation}
\end{subequations}

Carrying out the substitution $w_1=q^kx_1$ in $F(a,w)(1-q^kx_1/w_1)$ yields
\begin{equation}\label{e-ak}
A_k=\frac{\prod_{i=2}^n(x_1/x_i)_{a_1}(qx_i/x_1)_{a_i}
(q^{-k}x_i/x_1)_{a_i}^{-1}(x_1/w_i)_{a_1}^{-1}}
{(q^{-k-1})_{k+1}(q)_{a_1-k-1}}\times F(a^{(1)},w^{(1)}).
\end{equation}
Using \eqref{prop-a} with $(i,j,z)\mapsto (a_i,a_1,x_1/x_i)$,
we have
\[
\frac{(x_1/x_i)_{a_1}(qx_i/x_1)_{a_i}}
{(q^{-k}x_i/x_1)_{a_i}}
=q^{(k+1)a_i}\big(q^{-a_i}x_1/x_i\big)_{k+1}
\big(q^{k+1}x_1/x_i\big)_{a_1-k-1}.
\]
Substituting this into \eqref{e-ak} we obtain \eqref{A}.

Carrying out the substitution $w_1=q^jx_i$ in $F(a,w)(1-q^jx_i/w_1)$ yields
\begin{multline}\label{e-Bij}
B_{ij}=
\frac{(x_1/x_i)_{a_1}(qx_i/x_1)_{a_i}}{(q^{-j})_j(q)_{a_i-j-1}(q^{-j-1}x_1/x_i)_{a_1+1}}
\prod_{l=2}^{i-1}\frac{(x_l/x_i)_{a_l}(qx_i/x_l)_{a_i}}{(q^{-j}x_l/x_i)_{a_l}}
\prod_{l=i+1}^n\frac{(x_i/x_l)_{a_i}(qx_l/x_i)_{a_l}}{(q^{-j}x_l/x_i)_{a_l}}
\\
\times
\prod_{\substack{1\leq v<u\leq n\\v,u\neq i}}
(x_v/x_u)_{a_v}(qx_u/x_v)_{a_u}
\prod_{u=1}^n\prod_{v=2}^n(q^{-\chi(u=v)}x_u/w_v)_{a_u+\chi(u=v)}^{-1}.
\end{multline}
By \eqref{prop-b1} with $(i,j,z,k)\mapsto (a_1,a_i,x_i/x_1,j)$,
\begin{subequations}\label{e-B}
\begin{equation}
\frac{(x_1/x_i)_{a_1}(qx_i/x_1)_{a_i}}{(q^{-j-1}x_1/x_i)_{a_1+1}}
=-q^{(a_1+1)j+1}x_ix_1^{-1}\big(q^{1-a_1}x_i/x_1\big)_j\big(q^{j+2}x_i/x_1\big)_{a_i-j-1}.
\end{equation}
Using \eqref{prop-b2} and \eqref{prop-a} with $(i,j,z,k)\mapsto (a_l,a_i,x_i/x_l,j)$,
we have
\begin{equation}
\frac{(x_l/x_i)_{a_l}(qx_i/x_l)_{a_i}}{(q^{-j}x_l/x_i)_{a_l}}
=q^{ja_l}\big(q^{1-a_l}x_i/x_l\big)_j\big(q^{j+1}x_i/x_l\big)_{a_i-j},
\end{equation}
and
\begin{equation}
\frac{(x_i/x_l)_{a_i}(qx_l/x_i)_{a_l}}{(q^{-j}x_l/x_i)_{a_l}}
=q^{(j+1)a_l}\big(q^{-a_l}x_i/x_l\big)_{j+1}\big(q^{j+1}x_i/x_l\big)_{a_i-j-1}
\end{equation}
respectively.
\end{subequations}
Substituting \eqref{e-B} into \eqref{e-Bij} yields \eqref{B}.
\end{proof}

\section{Proof of Theorem~\ref{thm-1}}\label{sec-proof}
In this section, by the splitting formula \eqref{e-Fsplit} for $F(a,w)$ we obtain a recursion for 
$D_{v,\lambda}(a)$ if $\lambda_1\geq \max\{v_i\mid i=1,\dots,n\}$. Using the recursion we complete the proof of Theorem~\ref{thm-1}.

We begin by giving an easy result deduced from the $q$-binomial theorem.
\begin{prop}\label{prop-sum}
Let $n$ and $t$ be nonnegative integers. Then
\begin{equation}\label{e-sum}
\sum_{k=0}^{t}
\frac{q^{k(n-t)}}
{(q^{-k})_{k}(q)_{t-k}}
=\qbinom{n}{t}.
\end{equation}
\end{prop}
\begin{proof}
We can rewrite the left-hand side of \eqref{e-sum} as
\begin{equation}\label{e-sum2}
\sum_{k=0}^{t}
\frac{(-1)^kq^{k(n-t)+\binom{k+1}{2}}}{(q)_{k}(q)_{t-k}}
=\frac{1}{(q)_t} \sum_{k=0}^{t}q^{\binom{k}{2}}\qbinom{t}{k}(-q^{n-t+1})^k.
\end{equation}
By the well-known $q$-binomial theorem \cite[Theorem 3.3]{andrew-qbinomial}
\[
(z)_t=\sum_{k=0}^tq^{\binom{k}{2}}\qbinom{t}{k}(-z)^k
\]
with $z\mapsto q^{n-t+1}$, the right-hand side of \eqref{e-sum2} becomes
\[
(q^{n-t+1})_t/(q)_t.
\]
This is $\qbinom{n}{t}$, completing the proof.
\end{proof}

We utilize the splitting formula \eqref{e-Fsplit} for $F(a,w)$ to obtain a recursion for $D_{v,\lambda}(a)$.
\begin{lem}\label{lem-rec}
Let $D_{v,\lambda}(a)$ be defined in \eqref{GDyson}.
For $\lambda_1\geq \max\{v_i\mid i=1,\dots,n\}$,
\begin{equation}\label{e-ind}
D_{v,\lambda}(a)=
\begin{cases}
q^{-\lambda_1}\qbinom{|a|+\lambda_1}{a_1}D_{v^{(1)},\lambda^{(1)}}(a^{(1)}) &\text{if $\lambda_1=v_1$,}\\[4mm]
0 &\text{if $\lambda_1>v_1$}.
\end{cases}
\end{equation}
\end{lem}
\begin{proof}
Substituting the splitting formula \eqref{e-Fsplit} for $F(a,w)$ into \eqref{e-relation} yields
\[
D_{v,\lambda}(a)
=\CT_{x,w}\bigg(\sum_{k=-1}^{a_1-1}\frac{x^{-v}w^{\lambda}A_k}{\big(1-q^kx_1/w_1\big)}
+\sum_{i\in I}\sum_{j=0}^{a_i-1}\frac{x^{-v}w^{\lambda}B_{ij}}{\big(1-q^jx_i/w_1\big)}\bigg),
\]
where $I=\{i\mid a_i\neq 0,\ i=2,\dots,n\}$.
In the above equation, expanding $1/\big(1-q^kx_1/w_1\big)$ and $1/\big(1-q^jx_i/w_1\big)$ as
\[
\sum_{l\geq 0}\big(q^kx_1/w_1\big)^l \quad \text{and} \quad
\sum_{l\geq 0}\big(q^jx_i/w_1\big)^l
\]
respectively, and taking constant term with respect to $w_1$, we have
\begin{equation}\label{e-2sums}
D_{v,\lambda}(a)
=\CT_{x,w^{(1)}}\bigg(\sum_{k=-1}^{a_1-1}q^{k\lambda_1}x^{-v}x_1^{\lambda_1}w_2^{\lambda_2}\cdots w_n^{\lambda_n} A_k
+\sum_{i\in I}\sum_{j=0}^{a_i-1}q^{j\lambda_1}x^{-v}x_i^{\lambda_1}w_2^{\lambda_2}\cdots w_n^{\lambda_n} B_{ij}\bigg).
\end{equation}
Since $\lambda_1\geq \max\{v_i\mid i=1,\dots,n\}$, and $B_{ij}$ is a power series in $x_i$ with no constant term for $i\in I$,
\[
\CT_{x_i}x^{-v}x_i^{\lambda_1}w_2^{\lambda_2}\cdots w_n^{\lambda_n} B_{ij}
=\frac{x_i^{\lambda_1-v_i}w_2^{\lambda_2}\cdots w_n^{\lambda_n}}{{x_1^{v_1}}\cdots x_{i-1}^{v_{i-1}}x_{i+1}^{v_{i+1}}\cdots x_n^{v_n}} B_{ij}\Big|_{x_i=0}
=0.
\]
Hence, \eqref{e-2sums} reduces to
\begin{equation}\label{e-reduce}
D_{v,\lambda}(a)=\CT_{x,w^{(1)}}\sum_{k=-1}^{a_1-1}q^{k\lambda_1}x^{-v}x_1^{\lambda_1}w_2^{\lambda_2}\cdots w_n^{\lambda_n} A_k.
\end{equation}
Since $A_k$ is a power series in $x_1$ for $-1\leq k\leq a_1-1$,
if $\lambda_1>v_1$ then
\[
\CT_{x_1}x^{-v}x_1^{\lambda_1}w_2^{\lambda_2}\cdots w_n^{\lambda_n} A_k=
\frac{x_1^{\lambda_1-v_1}w_2^{\lambda_2}\cdots w_n^{\lambda_n}}{{x_2^{v_2}}\cdots x_n^{v_n}}A_k\Big|_{x_1=0}
=0.
\]
Consequently, $D_{v,\lambda}(a)=0$.

If $v_1=\lambda_1$, then by \eqref{e-reduce}
\[
D_{v,\lambda}(a)=\CT_{x,w^{(1)}}\sum_{k=-1}^{a_1-1}q^{k\lambda_1}\frac{w_2^{\lambda_2}\cdots w_n^{\lambda_n}}{x_2^{v_2}\cdots x_n^{v_n}} A_k.
\]
By the fact that $A_k$ is a power series in $x_1$ for $-1\leq k\leq a_1-1$ again,
\[
\CT_{x_1}A_k=A_k|_{x_1=0}.
\]
Hence,
\[
D_{v,\lambda}(a)=\CT_{x^{(1)},w^{(1)}}\sum_{k=-1}^{a_1-1}q^{k\lambda_1}\frac{w_2^{\lambda_2}\cdots w_n^{\lambda_n}}{x_2^{v_2}\cdots x_n^{v_n}} A_k\big|_{x_1=0}.
\]
By the expression \eqref{A} for $A_k$ and carrying out the substitution $x_1=0$ in $A_k$, we obtain
\begin{align}
D_{v,\lambda}(a)&=\sum_{k=-1}^{a_1-1}
\frac{q^{k\lambda_1+(k+1)(|a|-a_1)}}
{(q^{-k-1})_{k+1}(q)_{a_1-k-1}}
\times \CT_{x^{(1)},w^{(1)}}\frac{w_2^{\lambda_2}\cdots w_n^{\lambda_n}}{x_2^{v_2}\cdots x_n^{v_n}}F(a^{(1)},w^{(1)})\nonumber \\
&=q^{-\lambda_1}\sum_{k=-1}^{a_1-1}
\frac{q^{(k+1)(|a|+\lambda_1-a_1)}}
{(q^{-k-1})_{k+1}(q)_{a_1-k-1}}
\times \CT_{x^{(1)},w^{(1)}}\frac{w_2^{\lambda_2}\cdots w_n^{\lambda_n}}{x_2^{v_2}\cdots x_n^{v_n}}F(a^{(1)},w^{(1)})\nonumber \\
&=q^{-\lambda_1}\sum_{k=0}^{a_1}
\frac{q^{k(|a|+\lambda_1-a_1)}}
{(q^{-k})_k(q)_{a_1-k}}
\times \CT_{x^{(1)},w^{(1)}}\frac{w_2^{\lambda_2}\cdots w_n^{\lambda_n}}{x_2^{v_2}\cdots x_n^{v_n}}F(a^{(1)},w^{(1)}).\label{e-s}
\end{align}
By \eqref{e-sum} with $t\mapsto a_1$ and $n\mapsto |a|+\lambda_1$,
we find
\begin{subequations}\label{e-s1}
\begin{equation}
\sum_{k=0}^{a_1}
\frac{q^{k(|a|+\lambda_1-a_1)}}
{(q^{-k})_k(q)_{a_1-k}}
=\qbinom{|a|+\lambda_1}{a_1}.
\end{equation}
Using \eqref{e-relation} with
$\lambda\mapsto \lambda^{(1)}$, $w\mapsto w^{(1)}$, $v\mapsto v^{(1)}$
and $a\mapsto a^{(1)}$, we have
\begin{equation}
\CT_{x^{(1)},w^{(1)}}\frac{w_2^{\lambda_2}\cdots w_n^{\lambda_n}}{x_2^{v_2}\cdots x_n^{v_n}}F(a^{(1)},w^{(1)})
=D_{v^{(1)},\lambda^{(1)}}(a^{(1)}).
\end{equation}
\end{subequations}
Substituting \eqref{e-s1} into \eqref{e-s} yields
\[
D_{v,\lambda}(a)=
q^{-\lambda_1}\qbinom{|a|+\lambda_1}{a_1}D_{v^{(1)},\lambda^{(1)}}(a^{(1)}),
\]
completing the proof.
\end{proof}
Lemma~\ref{lem-rec} implies Theorem~\ref{thm-1} in a few steps.
\begin{proof}[Proof of Theorem~\ref{thm-1}]

If $v=\lambda$, then by repeated use of the nonzero part of \eqref{e-ind} we have
\[
D_{\lambda,\lambda}(a)=q^{-|\lambda|}\prod_{i=1}^n\qbinom{a_i+\cdots+a_n+\lambda_i}{a_i}.
\]

If $v\prec \lambda$, then
there exists an integer $k\in \{1,\dots,n-1\}$ such that
\[
v_1=\lambda_1,\dots,v_{k-1}=\lambda_{k-1},v_{k}<\lambda_{k},v_{j}\leq \lambda_{k} \quad \text{for all $j>k$.}
\]
It follows that for any $i\in\{1,2,\dots,k-1\}$,
\[
\lambda_i\geq \max\{v_j\mid j=i,\dots,n\}.
\]
Hence, we can use the nonzero part of \eqref{e-ind} $k-1$ times and obtain
\begin{equation}\label{rec-1}
D_{v,\lambda}(a)=q^{-\lambda_1-\cdots-\lambda_{k-1}}
\prod_{i=1}^{k-1}\qbinom{a_i+\cdots+a_n+\lambda_i}{a_i}D_{v^*,\lambda^*}(a^*),
\end{equation}
where $a^*=(a_k,\dots,a_n)$, $v^*=(v_k,\dots,v_n)$ and $\lambda^*=(\lambda_k,\dots,\lambda_n)$.
Since $v_k<\lambda_k$ and $v_{j}\leq \lambda_k$ for all $j>k$, by the zero part of \eqref{e-ind} with $(v,\lambda,a)\mapsto (v^*,\lambda^*,a^*)$,
we find
\[
D_{v^*,\lambda^*}(a^*)=0.
\]
Substituting this into \eqref{rec-1} gives $D_{v,\lambda}(a)=0$ for $v\prec \lambda$.
\end{proof}

\section{Examples beyond Theorem~\ref{thm-1}}\label{sec-beyond}

In this section we discuss a relation between $\widetilde{D}_{v,\lambda}(a)$ and $D_{v,\lambda}(a)$,
and some cases of $D_{v,\lambda}(a)$ when $v\npreceq \lambda$.

An application of Theorem~\ref{thm-1} is to deduce a closed-form expression for Kadell's generalized $q$-Dyson constant term $\widetilde{D}_{v,v^+}(a)$. In fact, this is our original motivation to investigate $D_{v,\lambda}(a)$. We give an example to show how to use Theorem~\ref{thm-1} and Kadell's ex-conjecture~\eqref{kadellconj} to obtain a closed-form expression for $\widetilde{D}_{(1,1),(1,1)}(a)$.
Here and in this following, we suppress the tails of zeros in $v$ like partitions.
For example, we write $(1,1)$ instead of $(1,1,0^{n-2})$.

As a special case of \cite[Chapter 1, Equation 5.10]{Mac95},
\[
h_r(x_1,\dots,x_n,y_1,\dots,y_m)=\sum_{k=0}^rh_k(x_1,\dots,x_n)h_{r-k}(y_1,\dots,y_m)
\]
for positive integers $r,n,m$.
Then we can expand $D_{(1,1),(1,1)}(a)$ as
\begin{align*}
&\CT_x
x_1^{-1}x_2^{-1}\Big(h_1\big(x^{(a)}\big)+h_1(q^{-1}x_1)\Big)
\Big(h_1\big(x^{(a)}\big)+h_1(q^{-1}x_2)\Big)
\times \prod_{1\leq i<j\leq n}
(x_i/x_j;q)_{a_i}(qx_j/x_i;q)_{a_j}\\
&=\CT_x
x_1^{-1}x_2^{-1}\Big(h_1\big(x^{(a)}\big)+q^{-1}x_1\Big)
\Big(h_1\big(x^{(a)}\big)+q^{-1}x_2\Big)
\times \prod_{1\leq i<j\leq n}
(x_i/x_j;q)_{a_i}(qx_j/x_i;q)_{a_j}\\
&=\widetilde{D}_{(1,1),(1,1)}(a)+q^{-1}\widetilde{D}_{(1),(1)}(a)
+q^{-1}\widetilde{D}_{(0,1),(1)}(a)+q^{-2}\widetilde{D}_{(0),(0)}(a).
\end{align*}
Then we can get
\[
\widetilde{D}_{(1,1),(1,1)}(a)
=D_{(1,1),(1,1)}(a)-q^{-1}\widetilde{D}_{(1),(1)}(a)
-q^{-1}\widetilde{D}_{(0,1),(1)}(a)-q^{-2}\widetilde{D}_{(0),(0)}(a).
\]
By Theorem~\ref{thm-1}, nonzero part of Kadell's ex-conjecture~\eqref{kadellconj},
and the $q$-Dyson constant term identity~\eqref{q-Dyson},
we can obtain an expression for $\widetilde{D}_{(1,1),(1,1)}(a)$, 
which turns out to be a sum (no longer a product). We omit the explicit expression for 
$\widetilde{D}_{(1,1),(1,1)}(a)$ here.

In \cite{cai}, Cai showed that the following property of $\widetilde{D}_{v,\lambda}(a)$ holds.
\begin{thm}\label{thm-Cai}
Let $v\in\mathbb{Z}^{n}$ and $\lambda$ be
a partition such that $|v|=|\lambda|$.
If $\widetilde{D}_{v,\lambda}(a)$ is non-vanishing, then
$v^{+}\geq \lambda$ in dominance order.
\end{thm}
By Theorem~\ref{thm-Cai} and the splitting formula \eqref{e-Fsplit} for $F(a,w)$, we can show that $D_{v,\lambda}(a)=0$ for some cases when $v\npreceq \lambda$.
For example, we can prove the next result.
\begin{ex}
For $n\geq 3$
\[
D_{(0,5,2),(4,3)}(a)=0.
\]
\end{ex}
In the above case, $v=(0,5,2)$, $\lambda=(4,3)$ and $v\npreceq \lambda$.
By \eqref{e-relation} we can write $D_{(0,5,2),(4,3)}(a)$ as
\[
D_{(0,5,2),(4,3)}(a)=\CT_{x,w}\frac{w_1^4w_2^3}{x_2^5x_3^2}F(a,w).
\]
Using the splitting formula \eqref{e-Fsplit} for $F(a,w)$ we have
\[
D_{(0,5,2),(4,3)}(a)=\CT_{x,w}\frac{w_1^4w_2^3}{x_2^5x_3^2}
\bigg(\sum_{k=-1}^{a_1-1}\frac{A_k}{1-q^kx_1/w_1}
+\sum_{i\in I}\sum_{j=0}^{a_i-1}\frac{B_{ij}}{1-q^jx_i/w_1}\bigg),
\]
where $I=\{i\mid a_i\neq 0,\ i=2,\dots,n\}$.
In the above, expanding $1/\big(1-q^kx_1/w_1\big)$ and $1/\big(1-q^jx_i/w_1\big)$ as
\[
\sum_{l\geq 0}\big(q^kx_1/w_1\big)^l \quad \text{and} \quad
\sum_{l\geq 0}\big(q^jx_i/w_1\big)^l
\]
respectively, and taking constant term with respect to $w_1$ yields
\begin{equation}\label{ex-1}
D_{(0,5,2),(4,3)}(a)=\CT_{x,w^{(1)}}\bigg(\sum_{k=-1}^{a_1-1}\frac{q^{4k}x_1^4w_2^3A_k}{x_2^5x_3^2}
+\sum_{i\in I}\sum_{j=0}^{a_i-1}\frac{q^{4j}x_i^4w_2^3B_{ij}}{x_2^5x_3^2}\bigg).
\end{equation}
Recall that $A_k$ is a power series in $x_1$, and $B_{ij}$ is a power series in $x_i$ with no constant term. Then
\[
\CT_{x_1}\frac{x_1^4w_2^3A_k}{x_2^5x_3^2}=0 \quad \text{for $-1\leq k\leq a_1-1$,}
\]
and
\[
\CT_{x_i}\frac{x_i^4w_2^3B_{ij}}{x_2^5x_3^2}=0 \quad
\text{for $3\leq i\leq n$ and $i\in I$.}
\]
Hence, if $a_2>0$ then \eqref{ex-1} reduces to
\begin{equation}\label{e-ex1}
D_{(0,5,2),(4,3)}(a)=\CT_{x,w^{(1)}}\sum_{j=0}^{a_2-1}\frac{q^{4j}x_2^4w_2^3B_{2j}}{x_2^5x_3^2}
=\sum_{j=0}^{a_2-1}q^{4j}\CT_{x,w^{(1)}}\frac{w_2^3B_{2j}}{x_2x_3^2},
\end{equation}
and $D_{(0,5,2),(4,3)}(a)=0$ otherwise. 

In the following we show that $D_{(0,5,2),(4,3)}(a)$ also vanishes for $a_2>0$. By the fact that $B_{2j}$ is a power series in $x_2$ with no constant term,
\[
\CT_{x_2}x_2^{-1}B_{2j}=\big(x_2^{-1}B_{2j}\big)\big|_{x_2=0}.
\]
Taking constant term of \eqref{e-ex1} with respect to $x_2$ and using the above equation, we have
\[
D_{(0,5,2),(4,3)}(a)
=\sum_{j=0}^{a_2-1}q^{4j}\CT_{x^{(2)},w^{(1)}}\frac{w_2^3}{x_3^2}\big(x_2^{-1}B_{2j}\big)\big|_{x_2=0}.
\]
Substituting the expression \eqref{B} for $B_{ij}$ with $i=2$, and then carrying out the substitution $x_2=0$ in $x_2^{-1}B_{2j}$ gives
\begin{multline}\label{e-ex3}
D_{(0,5,2),(4,3)}(a)
=C_{a_2}
\CT_{x^{(2)},w^{(1)}}\frac{w_2^3}{x_1x_3^2}
\prod_{\substack{u=1\\u\neq 2}}^n\prod_{v=2}^n
(q^{-\chi(u=v)}x_u/w_v)_{a_u+\chi(u=v)}^{-1}\\
\times\prod_{\substack{1\leq v<u\leq n\\v,u\neq 2}}
(x_v/x_u)_{a_v}(qx_u/x_v)_{a_u},
\end{multline}
where
\[
C_{a_2}=\sum_{j=0}^{a_2-1}\frac{-q^{(a_1+5)j+1+(j+1)\sum_{l=3}^na_l}}{(q^{-j})_j(q)_{a_2-j-1}}.
\]
For $i=3,\dots,n$, the right-hand side of \eqref{e-ex3} is a power series in
$w_i^{-1}$,
we can easily extract out the constant term with respect to $w_i$ by setting $w_i^{-1}=0$.
Then
\begin{equation}\label{e-ex2}
D_{(0,5,2),(4,3)}(a)=C_{a_2}
\CT_{x^{(2)},w_2}\frac{w_2^3}{x_1x_3^2}
\prod_{\substack{u=1\\u\neq 2}}^n(x_u/w_2)_{a_u}^{-1}
\prod_{\substack{1\leq v<u\leq n\\v,u\neq 2}}
(x_v/x_u)_{a_v}(qx_u/x_v)_{a_u}.
\end{equation}
By the generating function \eqref{e-gfcomplete} of complete symmetric functions,
\begin{equation}\label{hX}
\CT_{w_2} w_2^3 \prod_{\substack{u=1\\u\neq 2}}^n(x_u/w_2)_{a_u}^{-1}
=h_3(X),
\end{equation}
where $X=(x_1,qx_1,\dots,q^{a_1-1}x_1,x_3,qx_3,\dots,q^{a_3-1}x_3,\dots,
x_n,qx_n,\dots,q^{a_n-1}x_n)$ is an alphabet of cardinality $|a|-a_2$.
Substituting \eqref{hX} into \eqref{e-ex2} yields
\[
D_{(0,5,2),(4,3)}(a)=C_{a_2}\CT_{x^{(2)}}\frac{h_3(X)}{x_1x_3^2}
\prod_{\substack{1\leq v<u\leq n\\v,u\neq 2}}
(x_v/x_u)_{a_v}(qx_u/x_v)_{a_u}.
\]
By the definition of $\widetilde{D}_{v,\lambda}(a)$ in \eqref{GDyson1},
the constant term in the above equation is $\widetilde{D}_{(1,2),(3)}(a^{(2)})$.
Thus
\[
D_{(0,5,2),(4,3)}(a)=C_{a_2}\widetilde{D}_{(1,2),(3)}(a^{(2)}).
\]
Since $(3)>(1,2)^+=(2,1)$, by Theorem~\ref{thm-Cai}
the constant term $\widetilde{D}_{(1,2),(3)}(a^{(2)})=0$.
Hence $D_{(0,5,2),(4,3)}(a)=0$ for $a_2>0$.

In conclusion, $D_{(0,5,2),(4,3)}(a)=0$. 

Similar to the proof of $D_{(0,5,2),(4,3)}(a)=0$, we can show that $D_{(5,0,2),(4,3)}(a)=0$.
On the other hand, using Maple we can verify that
$D_{(5,2,0),(4,3)}(a)$, $D_{(2,0,5),(4,3)}(a)$, $D_{(2,5,0),(4,3)}(a)$
and $D_{(0,2,5),(4,3)}(a)$ do not vanish.
In general, for $v\npreceq \lambda$ we can not get a closed-form formula for
$D_{v,\lambda}(a)$.

\end{document}